\documentclass[12pt,a4paper]{article}
\usepackage{amsmath,amsthm,amssymb,amsfonts}
\usepackage{hyperref}
\usepackage{graphicx,subcaption}
\usepackage{hyperref}

\usepackage{enumitem}

\newcommand{\C}{{\mathbb{C}}}          
\newcommand{\N}{{\mathbb{N}}}          
\newcommand{\R}{{\mathbb{R}}}          
\newcommand{\Z}{{\mathbb{Z}}}          
\newcommand{\Sphere}{\mathbb{S}}

\newcommand{\rr}{\rightarrow}
\newcommand{\lrr}{\longrightarrow}

\newcommand{\na}{{\nabla}}

\newcommand{\cotg}{{\mathrm{cotg}}\,}

\newcommand{\dx}{{\mathrm{d}}}

\newcommand{\papa}[2]{\frac{\partial#1}{\partial#2}}

\newcommand{\cz}{{\overline{z}}}
\newcommand{\cf}{{\overline{f}}}

\newcommand{\ch}{{\overline{h}}}

\newcommand{\cinf}[1]{{\mathrm{C}}^\infty_{#1}}
\newcommand{\vol}{{\mathrm{vol}}}

\newtheorem{teo}{Theorem}
\newtheorem{prop}{Proposition}

\def\cyclic{\mathop{\kern0.9ex{{+}
\kern-2.2ex\raise-.28ex\hbox{\Large\hbox{$\circlearrowright$}}}}\limits}

\title{Vector fields with big and small volume on the 2-sphere}

\author{Rui Albuquerque}

\begin{document}


\maketitle


\begin{abstract}

We consider the problem of minimal volume vector fields on a given Riemann surface, specialising on the case of $M^\star$, that is, the arbitrary radius 2-sphere with two antipodal points removed. We discuss the homology theory of the unit tangent bundle $(T^1M^\star,\partial T^1M^\star)$ in relation with calibrations and a certain minimal volume equation. A particular family $X_{\mathrm{m},k},\:k\in\N$, of minimal vector fields on $M^\star$ is found in an original fashion. The family has unbounded volume, $\lim_k\vol({X_{\mathrm{m},k}}_{|\Omega})=+\infty$, on any given open subset $\Omega$ of $M^\star$  and indeed satisfies the necessary differential equation for minimality. Another vector field $X_\ell$ is discovered on a region $\Omega_1\subset \Sphere^2$, with volume smaller than any other known \textit{optimal} vector field restricted to $\Omega_1$.

\end{abstract}


\vspace*{1mm}

\vspace*{6mm}

\setcounter{section}{1}

\markright{\sl\hfill  Rui Albuquerque \hfill}

\vspace*{2mm}
\begin{center}
\begin{large}\textbf{{{1 -- Previous results}}}
\end{large}
\end{center}
\vspace*{2mm}



In this article we explore new ideas and examples of the theory of the volume of vector fields, in the continuation of the results in \cite{Alb2021} but of a different sort.


Suppose we are given an oriented Riemann surface $M$ possibly with boundary. Let $X$ be a unit norm $\mathrm{C}^2$ vector field on $M$. By definition, the volume of $X$ is, cf. \cite{GilMedranoLlinaresFusterSecond}:
\begin{equation}  \label{Definition_volume}
\vol(X) =\vol(M,X^*g^\mathrm{S})=\int_M\sqrt{1+\|\na_{e_0}X\|^2+\|\na_{e_1}X\|^2}\,\vol_M
\end{equation}
where $g^\mathrm{S}$ is the Sasaki metric on the unit tangent bundle $T^1M\rr M$ and $e_0,e_1$ is \textit{any} local orthonormal frame on $M$.

We recall that a unit vector field is a critical point of the volume functional if and only if the corresponding submanifold of $T^1M$ is minimal, cf. \cite{ GilMedranoLlinaresFusterSecond}.

In article \cite{Alb2021} one discovers a sufficient condition to have a minimal volume vector field. First, let $A=A_1+\sqrt{-1}A_0:M\lrr\C$ be the function defined by the components of $\na X$ in the direct orthonormal frame $\{X,Y\}$ on $M$. In other words, let $A_0,A_1$ be the functions defined by
\begin{equation}
 A_0=\langle\na_XX,Y\rangle\qquad\qquad A_1=\langle\na_YX,Y\rangle.
\end{equation}
Then, if $X$ satisfies the following differential equation in a conformal chart $z$ of $M$:
\begin{equation} \label{minimalvectorfieldcalibrated_intro}
 \papa{}{\cz}\frac{A}{\sqrt{1+|A|^2}}=0,
\end{equation}
then $X$ has minimal volume over its domain. We remark this equation is reminiscent of the minimal area surface graph equation.

The Cauchy-Riemann system above is orientation invariant, because $A$ transforms accordingly (eg. if we change $X$ by $-X$, then $A$ becomes $-\overline{A}$). Just as well as it is invariant if the role of $X$ and $Y$ is permuted. That is consistent with all $(\cos\alpha)X+(\sin\alpha)Y$ having the same volume, for every constant $\alpha\in\R$, just as equation \eqref{Definition_volume} will show. More importantly, as the reader may notice below, is that each CR-equation is orientation invariant.

The present article is concerned with the geometry of vector fields on the 2-sphere with canonical metric. It brings up some surprising results, in the continuation of \cite{BorGil2006,BorGil2010,BritoChaconJohnson,BritoJackelineIcaroAdriana, BritoGomesGoncalves,GilMedranoLlinaresFusterSecond,Wieg} and of course \cite{Alb2021}.

Let us start by informing about equation \eqref{minimalvectorfieldcalibrated_intro} in the fortunate case of a hyperbolic metric. In \cite{Alb2021} we find a solution of \eqref{minimalvectorfieldcalibrated_intro} for $M$ with constant sectional curvature $K<0$. Actually, given any $M$ and $X$ such that $|A|$ is constant, then $A$ is constant and $M$ has constant sectional curvature $K=-|A|^2\leq0$. It follows that $\vol(X)=\sqrt{1-K}\,\vol(M)$. The result applies thus to any germ of the non-trivial 2-dimensional Lie group with left invariant metric and left invariant vector field $X$.

Equation \eqref{minimalvectorfieldcalibrated_intro} proves quite difficult to solve, even for the constant hyperbolic metric in isothermal coordinates. Uniqueness of solutions (up to some rigid rotation) remains an open question.

Notice the equation gives a sufficient condition for minimality. A necessary condition is deduced in \cite[p. 538]{GilMedranoLlinaresFusterSecond}. Due to O.~Gil-Medrano and E.~Llinares-Fuster, we now know that a minimal vector field satisfies the Euler-Lagrange equation:
\begin{equation}  \label{GilMedranoLlinaresFuster_Equation}
  X({A_0/\sqrt{1+|A|^2}})+Y({A_1/\sqrt{1+|A|^2}})=0,
\end{equation}
As well as the following reassuring new result.
\begin{prop}
  Cauchy-Riemann equation \eqref{minimalvectorfieldcalibrated_intro} implies Euler-Lagrange equation \eqref{GilMedranoLlinaresFuster_Equation} of the variational problem.
\end{prop}
\begin{proof}
 The differential of a function is $\C$-linear by definition. If we have a holomorphic function, then its differential vanishes in the direction of $X+\sqrt{-1}Y$, a complex multiple of $\partial_\cz$. By a straightforward computation, we see that the imaginary part of $\dx({A/\sqrt{1+|A|^2}})(X+\sqrt{-1}Y)=0$ yields \eqref{GilMedranoLlinaresFuster_Equation}.
\end{proof}

Clearly \eqref{GilMedranoLlinaresFuster_Equation} alone is far from giving the Cauchy-Riemann equations. The real part,
\begin{equation}  \label{realpartofRAequation}
  X({A_1/\sqrt{1+|A|^2}})-Y({A_0/\sqrt{1+|A|^2}})=0,
\end{equation}
\textit{should} be a sufficient condition for minimality, but this remains uncertain or seems to have solution only in hyperbolic metric. Convincing the reader that this may be so is also a motivation for this article.

On the manifold $\Sphere^2$ with the round metric, punctured at two antipodal points, it is known that a minimum of $\vol$ is attained, with the solution being a certain $X_0$ given in Proposition \ref{doiscasoscalculaveis}. One easily checks that $X_0$ does not satisfy \textit{our} equation \eqref{realpartofRAequation}, which is a local issue. This is coherent with the theory, since we have found a vector field, though in a smaller region of $\Sphere^2$, which has even less volume than $X_0$ in that region. We present it later, below.

Returning to the general case, we may take a conformal chart of $M$ to rewrite $A$ and possibly improve \eqref{minimalvectorfieldcalibrated_intro}, cf. \cite{Alb2021}. A complex coordinate $z=x+\sqrt{-1}y$ corresponds with isothermal coordinates, ie. a chart such that the Riemannian metric is given by $\lambda|\dx z|^2$ for some function $\lambda>0$.

A real vector field $X$ is given by
\begin{equation}
X=a\partial_x+b\partial_y=f\partial_z+\cf\partial_\cz
\end{equation}
where $f=a+\sqrt{-1}b$ and where $\partial_x=\frac{\partial}{\partial x}$ and $\partial_z=\frac{1}{2}(\partial_x-\sqrt{-1}\partial_y)$ and $\partial_\cz=\overline{\partial_z}$. If $Z=h\partial_z+\ch\partial_\cz$ is also a vector field, then
\begin{equation}\label{metricconformal}
 \langle X,Z\rangle=\frac{\lambda}{2}(f\ch+\cf h).
\end{equation}
In particular $\|X\|^2=\lambda|f|^2$. Note that $Y=\sqrt{-1}(f\partial_z-\cf\partial_\cz)
=\overline{Y}$.

The Levi-Civita connection is given by
$\na_{\partial_z}\partial_z=\na_z\partial_z=\Gamma\partial_z$, where
$\Gamma=\frac{1}{\lambda}\papa{\lambda}{z}$, 
$\na_z\partial_\cz=\na_\cz\partial_z=0$ and
$\na_\cz\partial_\cz=\frac{1}{\lambda}\papa{\lambda}{\cz}\partial_\cz$. In particular, we
have $R(\partial_z,\partial_\cz)\partial_z=-\papa{\Gamma}{\cz}\partial_z$ and 
thus
\begin{equation}
K=\frac{\langle R(\partial_z,\partial_\cz)\partial_z,\partial_\cz\rangle}{\langle\partial_z,\partial_\cz\rangle^2}=-\frac{2}{\lambda}\papa{\Gamma}{\cz}=-\frac{2}{\lambda}\papa{^2\log\lambda}{z\partial\cz}.
\end{equation}
We have proved in \cite{Alb2021} that, in the case of a unit vector field,
\begin{equation}
A=-2\lambda f^2\papa{\cf}{z}=2(\Gamma f+\papa{f}{z}) .
\end{equation}
Since $|A|=2|\papa{f}{\cz}|$, we have that $X$ is holomorphic if and only if it is parallel.

\vspace*{2mm}
\begin{center}
\begin{large}\textbf{{2 -- Topological invariants on the boundary}}
\end{large}
\end{center}
\vspace*{2mm}

We have found the equations of a 2-form calibration $\varphi\in\Omega_{T^1M}^2$ having the minimal vector fields as calibrated submanifolds.

The topology of a vector field determines a \textit{class} of its volume. This is what we shall deduce empirically and what was claimed in \cite{GluckZiller} regarding $\Sphere^3$, with no concerns on singularities and the class being the homology class in $T^1\Sphere^3$. This line of ideas was continued in \cite{Ped}, where one vector field $W$ with singularity on a hypersphere is evenly associated to a certain homology class.

There is now some light from the analytic theory of calibrations to further clarify the path between homology and minimal volume in other dimensions.

Let us develop these ideas, arguing first in any dimension. Let $M$ be a compact Riemannian $(n+1)$-manifold possibly with boundary. Then we have the following isomorphism of Poincar\'e-Lefschetz duality, cf. \cite[Section 3.3]{Hatcher}, with integer coefficients:
\begin{equation} \label{P-L_duality}
  H_{k}(T^1M,\partial T^1M)\simeq H^{2n+1-k}(T^1M)({\simeq} H^{k}(T^1M)) .
\end{equation}
The second isomorphism may be Hodge duality, recalling that de\,Rham and
singular cohomologies coincide for manifolds. However, we do not know of a precise
statement for Hodge duality for manifolds with boundary. (On the other hand, there exists
a connecting homomorphism in the middle degree $k=n+1$.)

Let $M^\star$ denote a given closed manifold $M$ \textit{with} a finite number
of points $p_1,\ldots,p_N$ removed. Notice that $T^1M^\star$ is not a manifold with boundary, for it corresponds with the removal of a disjoint collection of spheres. So let $M_\epsilon$ denote the manifold with boundary
$M\backslash\cup_iB_\epsilon(p_i)$, where the open geodesic balls do not intersect. Since $M_\epsilon\subset M_{\epsilon_1}$ for $\epsilon_1<\epsilon$, the cohomology rings $H^j(T^1M_{\epsilon_1})$ are well-defined and isomorphic between them, $\forall \epsilon_1$.
By \cite[Proposition 3.33]{Hatcher}, there exist the following two inductive limits
\begin{equation} \label{limitindurelativehomology}
\varinjlim H^{2n+1-k}(T^1M_\epsilon) {\simeq}\varinjlim H^{k}(T^1M_\epsilon) {\simeq} H^k(T^1M^\star).
\end{equation}
We conjecture the above to be all isomorphic.

By the first isomorphism in \eqref{P-L_duality}, we have:
\begin{equation}
H_{k}(T^1M^\star,\partial T^1M^\star):= \varinjlim H_{k}(T^1M_\epsilon,\partial T^1M_\epsilon) {\simeq} H^k(T^1M^\star) .
\end{equation}
So we never overcome the uncertainty of Hodge duality with boundary.


We return to dimension 2. It is clear that $\partial T^1M_\epsilon$ is a disjoint union of $N$ 2-torus. Using parallel translation to a base point $\Sphere^1$-fibre along each circle, any $\mathrm{C}^2$ unit vector field $X$ defined on $M^\star$
certainly has its degree in $\Z$ as it restricts to a map from $\partial
B_\epsilon(p_i)=\Sphere^1$ to itself. The field $X$ is not {singular}
at the $p_i$ in the sense of having a zero; though it still has an index,  $I_X(p_i)$, independent of $\epsilon$. The sum $\sum_{i=1}^NI_X(p_i)$ is the Euler characteristic of the surface (Poincaré-Hopf Theorem).

The vector fields $X$ determine a class in the relative homology $H_2(T^1M_\epsilon,\partial T^1M_\epsilon)$
in principle dependent on the various indices at the $p_i,\ i=1,\ldots,N$ and nothing else. Notice $X(M^\star)$ sits in $T^1M^\star$ transversely to $\partial T^1M^\star\subset T^1M$. So the homology class of the field is the main invariant; the relative homology must not complicate much more.

What one would hopefully expect from calibrations is that each $[\varphi]\in
H^2(T^1M^\star)$ determines a class $[X(M^\star)]$ of \textit{minimal volume} in (relative) homology. On the other hand, Theorem 1 in \cite{Alb2021} has led to quite demanding solutions, besides the trivial case for hyperbolic space.

Extending the theory to complex vector fields $X\in\Gamma(M;TM\otimes\C)$ could perhaps yield a path to the necessary and sufficient condition for the minimal vector field.

Finally, one may consider a Berger metric type dilation on the unit tangent
bundle, ie. the usual metric on $T^1M$ with a weight on the direction of the geodesic spray $e_0$ (our usual notation). Thus we assume the vector field $\tilde{e}_0=\mu e_0$ on the manifold $T^1M$, with $\mu\in\cinf{M}(\R^+)$, has unit norm and its orthogonal plane remains \textit{the same}. The problem of finding minimal vector fields for $\vol_X$ would depend on
{being able to} optimize volume with this known metric. In other words, to
be certain of minimizing through the right weight function $\mu>0$ and the
minimal vector field of $\int_M\frac{1}{\mu}\sqrt{1+A_{1}^2+\mu^2A_{0}^2}\,\vol_M$. This strategy is equivalent to that referred in \cite[Remark in Section 3]{BorGil2010} and \cite{GilMedranoLlinaresFusterSecond} for spheres.

\vspace*{2mm}
\begin{center}
\begin{large}\textbf{{3 -- Vector fields on the sphere: meridian type } }
\end{large}
\end{center}
\vspace*{2mm}


We consider the radius $r$ sphere with two antipodal points removed
\begin{equation}
 M^\star= \Sphere^2\backslash\{p_S,p_N\}
\end{equation}
endowed with the round metric $\langle\ ,\ \rangle$, ($x=\log(\tan(\theta/2))$)
\begin{equation}
 \langle\ ,\ \rangle=r^2(\dx\theta)^2+r^2\sin^2\theta(\dx\phi)^2=r^2\sin^2\theta((\dx x)^2+(\dx\phi)^2),
\end{equation}
where $r>0$ is constant, and $(\theta,\phi)\in D=]0,\pi[\times[0,2\pi[$. Letting $i=\sqrt{-1}$ and continuing as in \eqref{metricconformal}, we have $z=x+i\phi$ and then $\sin\theta\dx x={\dx \theta}$ and  $\dx z=\dx x+i\dx\phi$. Hence
\begin{equation}
 \lambda=r^2\sin^2\theta
 \end{equation}
and
\begin{equation}
 \partial_x=\sin\theta\,\partial_\theta,\ \ \qquad
 \partial_z=\frac{1}{2}(\partial_x-i\partial_\phi) .
 \end{equation}
 In particular the volume form is given by
\begin{equation}
 \lambda\,\dx x\wedge\dx\phi=\frac{i}{2}\lambda\,\dx z\wedge\dx\cz= r^2\sin\theta\,\dx\theta\wedge\dx\phi.
\end{equation}
We have $\Gamma=\frac{1}{\lambda}\partial_z\lambda=\frac{1}{2\sin\theta}
\partial_\theta\sin^2\theta=\cos\theta$, which verifies
\begin{equation}
K=-\frac{2}{\lambda}\papa{\Gamma}{\cz}=-\frac{\sin\theta}{r^2\sin^2\theta}
\partial_\theta{\cos\theta}=\frac{1}{r^2}.
\end{equation}

We also require the Levi-Civita connection in real coordinates
\begin{equation} \label{LCconnectionofS2}
 \na_\theta\partial_\theta=0,\qquad\na_\theta\partial_\phi=\na_\phi\partial_\theta=\cotg\theta\,\partial_\phi,\qquad
 \na_\phi\partial_\phi=-\cos\theta\sin\theta\,\partial_\theta .
\end{equation}

Guessing from the first equation above on the case of the unit norm vector field $\frac{1}{r}\partial_\theta$, we endeavour to look for those unit vector fields which are parallel along every meridian $p_Sp_N$.

Let $X=\frac{a}{r}\partial_\theta+\frac{b}{r\sin\theta}\partial_\phi$ with $a,b$
real valued $\mathrm{C}^2$ functions on $D$. Then we find
\begin{equation}
\begin{aligned}  \label{calculosauxiliares}
 \na_\theta X & = a'_\theta\partial_\theta+\frac{b'_\theta\sin\theta-b\cos\theta}{\sin^2\theta}\partial_\phi+\frac{b}{\sin\theta}\cotg\theta\,\partial_\phi \\  & = a'_\theta\partial_\theta+\frac{b'_\theta}{\sin\theta}\partial_\phi.
\end{aligned}
\end{equation}
Hence
 \begin{equation}
 \na_\theta X=0  \quad\Leftrightarrow\quad
 \begin{cases} a'_\theta=0,\\ b'_\theta=0 \end{cases} .
\end{equation}
Since for unit $X$ we must have $a^2+b^2=1$, there exists $\zeta=\zeta(\phi)$,
function only of $\phi$, such that
\begin{equation} \label{meridianapproachVF}
X=\frac{\cos\zeta}{r}\partial_\theta+ \frac{\sin\zeta}{r\sin\theta}\partial_\phi
\end{equation}

Let $Y=-\frac{\sin\zeta}{r}\partial_\theta+
\frac{\cos\zeta}{r\sin\theta}\partial_\phi$ be the unique vector field such that
$X,Y$ is a direct orthonormal frame. Routine computations yield:
\begin{equation}
 A_0=\langle\na_XX,Y\rangle= \sin\zeta\frac{\zeta'_\phi+\cos\theta}{r\sin\theta}  ,
\end{equation}
\begin{equation}
 A_1=\langle\na_YX,Y\rangle=\cos\zeta\frac{\zeta'_\phi+\cos\theta}{r\sin\theta}.
\end{equation}
Thus
\begin{equation}   \label{meridianVFmagnusfunction}
\frac{A_1+iA_0}{\sqrt{1+|A|^2}}=\frac{e^{i\zeta}(\zeta'_\phi+\cos\theta)}{\sqrt{r^2\sin^2\theta+(\zeta'_\phi+\cos\theta)^2}} .
\end{equation}
Notice that even with $r=1$ and $\zeta=0$, we get $\partial_\cz\cos\theta=-\frac{1}{2}\sin^2\theta$. (Of course, in a general setting, a vector field which has $A_0=0$ is not a good candidate as a solution of equation \eqref{minimalvectorfieldcalibrated_intro}, unless it is parallel.) One would expect that $\partial_\theta$ has minimal volume; and this is true globally, as it was proved by \cite{BritoChaconJohnson} and we shall soon recall.

On the radius $r$ sphere, we have:
 \begin{equation} \label{generalformulaforvoltypeI}
  \vol(X)=r\iint_D\sqrt{r^2\sin^2\theta+(\zeta'_\phi+\cos\theta)^2}\,\dx\theta\wedge\dx\phi.
 \end{equation}

The field $X$ is well defined and continuous on $M^\star$ if it satisfies
$\lim_{\phi\rr2\pi}\zeta(\phi)=\zeta(0)+2k\pi$, for some $k\in\Z$. We may further assume that $\zeta$ is $\mathrm{C}^2$ on a neighborhood of $[0,2\pi]$. This way the field $X$ becomes also $\mathrm{C}^2$.

The latter is the case when we fix $k\in\Z,\phi_0\in\R$ and take
\begin{equation}
\zeta=k\phi+\phi_0.
\end{equation}
We define these as the \textit{vector fields of meridian} or \textit{meridians type} ($r=1$):
\begin{equation}  \label{meridianVF}
X_{\mathrm{m},k}=\cos(k\phi+\phi_0)\partial_\theta+
\frac{\sin(k\phi+\phi_0)}{\sin\theta}\partial_\phi .
\end{equation}
Thus $k$ is the number of times that $X_{\mathrm{m},k}$ winds around itself when it goes around a  parallel or circle-of-latitude.

Now we are rewarded with a remarkable result.
\begin{prop} \label{Almostcertaincandidates}
 Every vector field of meridian type satisfies the Euler-Lagrange equation \eqref{GilMedranoLlinaresFuster_Equation} for a minimal volume vector field.
\end{prop}
\begin{proof}
 An easy way to see this is to note from \eqref{meridianVFmagnusfunction} that $A/\sqrt{1+|A|^2}$ equals $e^{i\zeta}f$ with $f=f(\theta)$ and $\zeta=k\phi+\phi_0$. Now, from \eqref{meridianapproachVF},
  \begin{align*}
  X+iY & =  \frac{\cos\zeta}{r}\partial_\theta+ \frac{\sin\zeta}{r\sin\theta}\partial_\phi -i \frac{\sin\zeta}{r}\partial_\theta+i \frac{\cos\zeta}{r\sin\theta}\partial_\phi  \\
  & =  \frac{e^{-i\zeta}}{r}\partial_\theta+i\frac{e^{-i\zeta}}{r\sin\theta}\partial_\phi .
   \end{align*}
   Hence
 \begin{align*}
  \dx(e^{i\zeta}f)(X+iY) & = \frac{e^{-i\zeta}}{r}\partial_\theta (e^{i\zeta}f)+i\frac{e^{-i\zeta}}{r\sin\theta}ike^{i\zeta}f \\
  &= \frac{1}{r}f'_\theta-\frac{k}{r\sin\theta}f .
 \end{align*}
 In particular, the imaginary part vanishes. (The real part does not.)
\end{proof}

From \eqref{generalformulaforvoltypeI}, we have the elliptic integral:
 \begin{equation} \label{volumeofMeridiantype}
  \vol(X_{\mathrm{m},k})=r\iint_D\sqrt{k^2+2k\cos\theta+r^2\sin^2\theta+\cos^2\theta}\,\dx\theta\wedge\dx\phi.
 \end{equation}
We notice the cases $k$ and $-k$ yield the same volume on $M^\star$, as expected.

The next result shows that we can have big volume everywhere. We admit $k\geq0$ for convenience, knowing that minor adaptation is needed for $k<0$.

\begin{teo}[Big volume everywhere]  \label{Everywherebigvolume}
The sequence of unit vector fields $X_{\mathrm{m},k},\ k\in\N$, defined on $M^\star$ is such that, for every open subset $\Omega\subset M^\star$ corresponding to a domain $\tilde{D}\subset D$, we have, in case $r\leq1$,
 \begin{equation}
   r\sqrt{k^2-2k+r^2}\vol_{\mathrm{Euc}}(\tilde{D})< \vol({X_{\mathrm{m},k}}_{|\Omega})<
   r(k+1)\vol_{\mathrm{Euc}}(\tilde{D})
 \end{equation}
 and, in case $r\geq1$,
  \begin{equation}
   r(k-1)\vol_{\mathrm{Euc}}(\tilde{D})<\vol({X_{\mathrm{m},k}}_{|\Omega})<r\sqrt{k^2+2k+r^2}\vol_{\mathrm{Euc}}(\tilde{D}).
 \end{equation}
  In particular, for every open sets $\Omega\subset \Omega_1\subset M^\star$,
 \begin{equation}
  \sup\bigl\{\vol(X_{|\Omega}): \ X\ \mbox{is a unit vector field on}\ \Omega_1\bigr\}=+\infty.
 \end{equation}
\end{teo}
\begin{proof}
In regard with \eqref{volumeofMeridiantype}, we have for instance for $r\leq1$ that
 \[ {k^2-2k+r^2} <{k^2+2k\cos\theta+r^2\sin^2\theta+\cos^2\theta} <k^2+2k+1  \]
 and so the result follows.
\end{proof}

A radius $r\neq1$ also brings \textit{stability} into discussion. This was foremost observed in \cite{BorGil2006,BritoGomesGoncalves} in general, hence from now on we assume $r=1$ for the meridians type vector fields.

Let us recall now the unit vector field $W$ on an $n$-dimensional punctured sphere, defined by S.~Pedersen in \cite{Ped} and denoted there by $W$. It is defined, on $\Sphere^2\backslash\{p_S\}$ only, as the parallel transport along the meridians of one chosen unit tangent vector at the North pole.

Clearly $W$ does not extend to the South pole and its volume does not depend on the initial choice. The following result becomes also geometrically evident.
\begin{prop}
 On $M^\star$ the vector field $X_{\mathrm{m},1}$ coincides with $W$.
\end{prop}
\begin{proof}
Both vector fields rotate, once and uniformly, while they go around the parallels, ie. the curves $\theta=$\,constant. Since the two fields are defined by parallel transport along the meridians, they must be the same, cf. Figure \ref{fig:meridianVF}.
\end{proof}

\begin{figure}
  \centering
    \includegraphics[width=.43\linewidth]{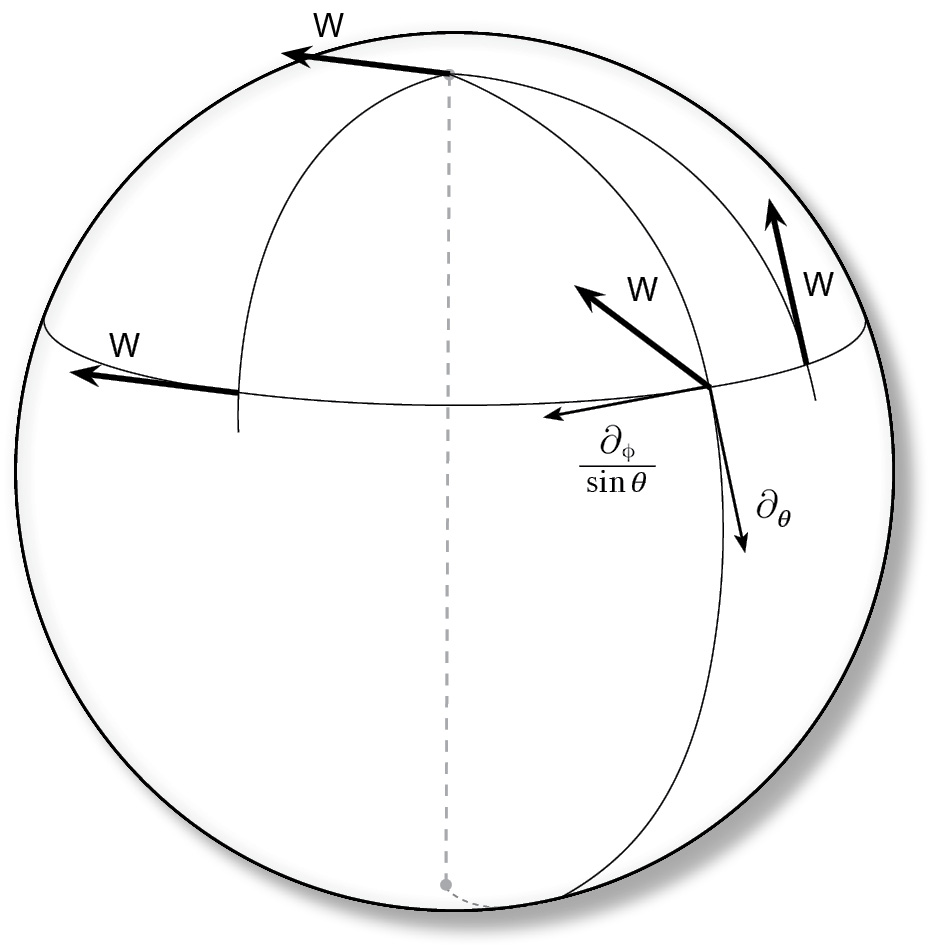}
    \caption{$W=X_{\mathrm{m},1}$.}
      \label{fig:meridianVF}
\end{figure}

Article \cite{Ped} is mostly concerned with $W$, referring it as a
{Pontryagin cycle} in the context of the homology theory of the orthogonal Lie groups and Stiefel manifolds. \cite{BorGil2010} also refers to $W$ as an example of a certain Pontryagin vector field.

Taking the cases $k=0,1$ with the above point of view, we have not found a reference for the vector fields $X_{\mathrm{m},k}$.

Formula \eqref{volumeofMeridiantype} also yields the computational part, with $k=0$ and $k=1$, of the next result.
\begin{prop}[\cite{BorGil2010,BritoChaconJohnson,Ped}]    \label{doiscasoscalculaveis}
 The meridians type vector fields $X_{\mathrm{m},0}=\cos\phi_0\partial_\theta+\frac{\sin\phi_0}{\sin\theta}\partial_\phi$\, and $X_{\mathrm{m},1}$ have minimal volume in the
respective homology classes $M^\star$ and $\Sphere^2\backslash\{p_S\}$, when $r=1$. Moreover,
 $\vol(X_{\mathrm{m},0})=2\pi^2\approx6.28\pi$ and $\vol(X_{\mathrm{m},1})=8\pi$.
\end{prop}
 Let us just show that
\begin{align*}
 &  \vol(X_{\mathrm{m},1})  =\int_0^{2\pi}\int_0^\pi\sqrt{2+2\cos\theta}\,\dx\theta\dx\phi= 2\sqrt{2}\pi\int_0^\pi\sqrt{1+\cos\theta}\,\dx\theta \\
 &  = 2\sqrt{2}\pi\int_{-1}^1\frac{\sqrt{1+t}}{\sqrt{1-t^2}}\,\dx t   = 2\sqrt{2}\pi\int_{-1}^1\frac{1}{\sqrt{1-t}}\,\dx t=-4\sqrt{2}\pi\bigl[\sqrt{1-t}\bigr]^1_{-1}=8\pi
\end{align*}
 is consistent with \cite[Theorem 10]{Ped}. Minimality in the homology class is proved in \cite{BorGil2010}.

\begin{teo}
 For the meridians type vector field $X=X_{\mathrm{m},k}$ with $k\geq0$, we may fix an orientation on $M^\star$ such that
\begin{equation}
     I_X(p_N)=1-k\qquad I_X(p_S)=1+k.
\end{equation}
\end{teo}
\begin{proof}
 We may fix an orientation on $D$ and deduce that $\partial_\theta$, or any other $X_{\mathrm{m},0}$, has index 1 at $p_N$. Indeed, fixing a trivialization of $T^1\Sphere^2$ in a neighbourhood of $p_N$, along any directed circle-of-latitude (col) around $p_N$ the vector field $\partial_\theta$ describes another entire circle identically, ie. describes the identity map of $\Sphere^1$ after parallel transport to a base point of the col, which therefore gives a degree 1 self-map of $\Sphere^1$. Taking a neighborhood of $p_S$, the field $\partial_\theta$ moves in the same way as before, even though the direction in any col close to $p_S$ must be the opposite of the previous homotopic col. Either indices of $\partial_\theta$ or $-\partial_\theta$ are $1$ at $p_S$.

 For $X_{\mathrm{m},1}=W$, we have that $p_N$ is a smooth point. Indeed, by definition, $W$  extends as a unit vector to that point (notice it does not extend continuously to the South pole). The index at $p_N$ is thus 0, and this follows also because, along any directed col around $p_N$, the vector field $W$ describes in its range a new circle rotating `clockwise' once, ie. it gives a constant valued self-map of $\Sphere^1$ after parallel transport to a base point of the col as above, which henceforth gives a degree 0 map (compare with $\partial_\theta$). Conversely, we reach the South pole with a degree 2 map. Indeed, conforming with orientation, one would have to rotate $W$ twice `anti-clockwise' to draw the symmetry with the North pole; hence the index at $p_S$ is $2$.

 Continuing this way, for $k\geq2$, we will find $I_X(p_N)=1-k$. And, still with $k>0$, proceeding to find the referred symmetry, achieved by unwinding our vector field $2k$ times, we find $I_X(p_S)=1-(k-2k)=1+k$. And the result follows.
\end{proof}
Notice we have $\chi(\Sphere^2)=2$ as predicted by the Theorem of Poincaré-Hopf.

Let us recall from \cite[Theorem 1.1]{BritoChaconJohnson} that any vector field $X$ on $M^\star$ with radius 1 satisfies:
\begin{equation}  \label{minoravolX}
 (\pi+|I_X(p_S)|+|I_X(p_N)|-2)2\pi\leq\vol(X) .
\end{equation}
We have thus verified the `big index, big volume' precept.

In the continuation of Proposition \ref{doiscasoscalculaveis}, the minimum volume $2\pi^2$ is attained with $k=0$, ie. with $X_{\mathrm{m},0}$ in its domain $M^\star$.

Minimality of $X_{\mathrm{m},1}$ in its domain and homological class is conjectured in \cite{Ped} and proven by \cite{BorGil2010}. From \eqref{minoravolX}, we only get $(\pi+2+0-2)2\pi=2\pi^2<8\pi$. Minimality depends on the topology of the domain and the vector fields. How the corresponding homological classes rule volume escapes to our understanding.

For any $k>0$, we have $ (\pi+1+k+k-1-2)2\pi=(\pi+2k-2)2\pi\leq\vol(X_{\mathrm{m},k})$.

Notice with $k\geq4$, we get $(\pi+2k-2)2\pi < (k-1)2\pi^2< \vol(X_{\mathrm{m},k})$ by Theorem \ref{Everywherebigvolume}. Hence \eqref{minoravolX} is sharp  for $k=0$ and not sharp for $k\geq4$.

Up to now the relative homology class of a vector field on $M^\star$ is determined by a unique index, namely, the integer $k$. Based on the known cases $k=0,1$ and Proposition \ref{Almostcertaincandidates}, we could finally state a conjecture:
 For each $k\in\Z^+$, the meridians type vector field $X_{\mathrm{m},k}$ realizes minimal volume in its (relative) homology class.

 However, last but not least, the question has been solved. Some of the previous results can also be found in \cite{BritoJackelineIcaroAdriana}, including the next theorem. Since $k>0$ is a topological invariant, the above conjecture fades away.
\begin{teo}[\cite{BritoJackelineIcaroAdriana}]
Let $Z$ be a unit vector field on $M^\star$ and $k+1=\max\{ I_Z(p_N),I_Z(p_S)\}$. Then
\begin{equation}
 \vol(Z)\geq\pi L(\varepsilon_{k+1}),
\end{equation}
where $L(\varepsilon_{k+1})$ is the length of the ellipse $\frac{x^2}{(k+1)^2}+\frac{y^2}{(k-1)^2} = 1$.
\end{teo}

\vspace*{2mm}
\begin{center}
\begin{large}\textbf{{4 -- New vector fields on the sphere: parallels type}}
\end{large}
\end{center}
\vspace*{2mm}

Finding solutions of the holomorphic equation \eqref{minimalvectorfieldcalibrated_intro} in $M^\star$ remains a local question. We shall see there are vector fields in an open region of $\Sphere^2$ with even less volume than the above.

We now consider the equations for unit vector fields $Z$ which are parallel along the \textit{parallels}; the latter being also known as the \textit{circles-of-latitude}.

We return to the radius $r$ 2-sphere.

Let $Z=\frac{a}{r}\partial_\theta+\frac{b}{r\sin\theta}\partial_\phi$ with $a^2+b^2=1$. Applying \eqref{LCconnectionofS2}, the desired condition on $Z$ translates into
\begin{equation}
   \na_\phi Z=0\quad\Leftrightarrow\quad \begin{cases}
   a'_\phi-b\cos\theta=0 \\ b'_\phi+a\cos\theta=0
          \end{cases}.
\end{equation}
The general solution follows:
\begin{equation}
 a=\sin\eta,\qquad b=\cos\eta
 \end{equation}
where
\begin{equation}
\eta(\theta,\phi)=\phi{\cos\theta}+\phi_0,\qquad \ \phi_0\in\R .
\end{equation}

Since $\eta(\theta,2\pi)-\eta(\theta,0)=2\pi\cos\theta\notin2\pi\Z$, it is only possible to have the vector field $Z$ defined on $M^\star\backslash\{\phi=0\}$, ie. $M^\star$ with one meridian removed.

Now we notice
\begin{equation}
\begin{cases}
     a'_\theta = -b{\phi\sin\theta} \\   b'_\theta=a{\phi\sin\theta}
     \end{cases}.
\end{equation}
As usual, we consider the unit orthogonal $Y=-\frac{b}{r}\partial_\theta+\frac{a}{r\sin\theta}\partial_\phi$. Recalling \eqref{calculosauxiliares}, we find
\begin{equation}
 \na_\theta Z =\frac{a'_\theta}{r}\partial_\theta+\frac{b'_\theta}{r\sin\theta}\partial_\phi  =  {\phi\sin\theta}\,Y.
\end{equation}
Hence
\begin{equation}
 A_0=\langle\na_ZZ,Y\rangle=\frac{a}{r}\langle\na_\theta Z,Y\rangle
 =\frac{a\phi\sin\theta}{r} ,\quad
 A_1=\langle\na_YZ,Y\rangle=-\frac{b\phi\sin\theta}{r}
\end{equation}
and
\begin{equation}
 \frac{A}{\sqrt{1+|A|^2}}=-e^{-i\eta}\frac{\phi\sin\theta}{\sqrt{r^2+\phi^2\sin^2\theta}}
\end{equation}
leading through easy computations to a conclusion.
\begin{prop}
 No integrability equation \eqref{GilMedranoLlinaresFuster_Equation} or \eqref{realpartofRAequation} is satisfied with $Z$ above.
\end{prop}

Finally, by \eqref{Definition_volume}, the volume of a circles-of-latitude vector field is
\begin{equation} \label{volumeofCirclesofLatitudetype}
\vol(X_\ell)=r\iint_D\sqrt{r^2+\phi^2\sin^2\theta}
 \sin\theta\,\dx\phi\dx\theta={r}
\int_0^\pi\int_0^{2\pi\sin\theta}\sqrt{ r^2+y^2 }\,\dx y\dx\theta .
\end{equation}
Indeed, we call the above the \textit{circles-of-latitude} or \textit{parallels type vector field} ($r=1$):
\begin{equation}
 X_{\ell}={\sin(\phi\cos\theta+\phi_0)}\partial_\theta+\frac{\cos(\phi\cos\theta+\phi_0)}{\sin\theta}\partial_\phi.
\end{equation}

Finally a new minimum of the volume functional is achieved inside $D$.
\begin{teo}  \label{Nunesregion}
 On the region $\Omega=\{ (\theta,\phi)\in D:\ \phi\neq0,\
\phi\sin^2\theta<|\cos\theta|\}$, a circles of latitude type vector field
has volume strictly lower than the minimal meridians type vector fields.
More precisely, on $\Omega$,
 \begin{equation}
 \vol(X_\ell)<\vol(\partial_\theta)=\vol_{\mathrm{Euc}}(\Omega).
 \end{equation}
\end{teo}
\begin{proof}
 Recall from formula \eqref{volumeofMeridiantype} with $k=0$ that the volume of
$\partial_\theta$ is equal to the Euclidean volume of the region of definition in $D$. Now the result is straightforward from \eqref{volumeofCirclesofLatitudetype} and
$(1+\phi^2\sin^2\theta)\sin^2\theta=\sin^2\theta+\phi^2\sin^4\theta<1$.
\end{proof}

\begin{figure}
  \centering
    \includegraphics[keepaspectratio,width=.43\linewidth]{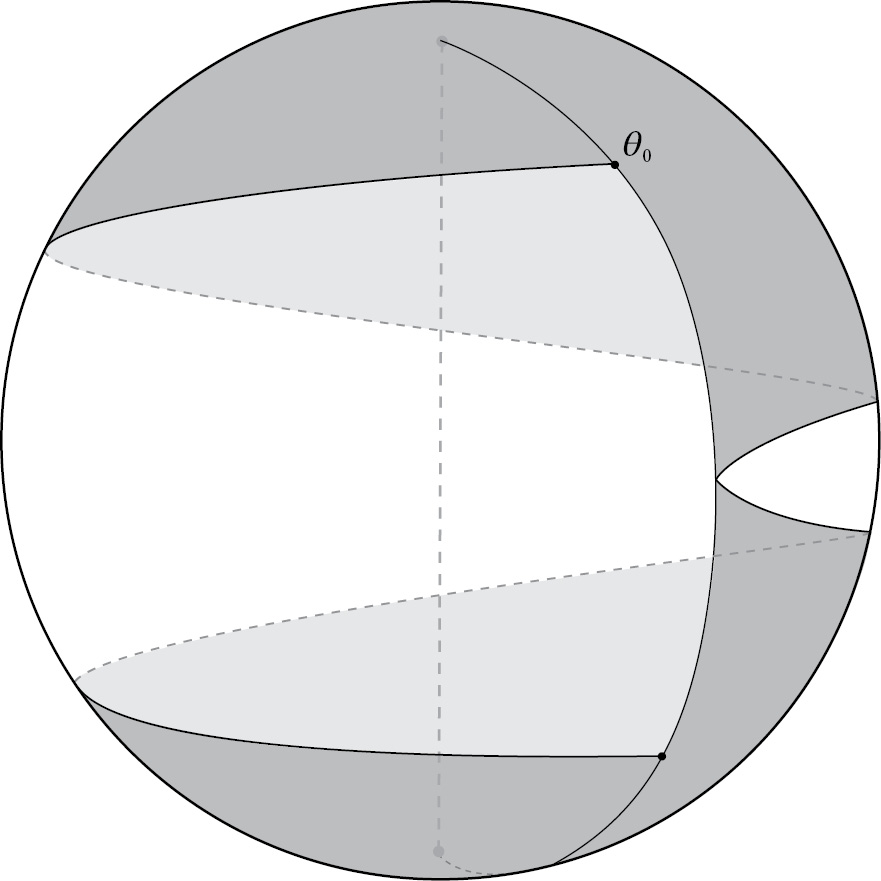}
    \caption{Domain where $\vol(X_\ell)<\vol(\partial_\theta)$.}
      \label{fig:Nunesspace}
\end{figure}

With $\theta\in]0,\theta_0[\cup]\pi-\theta_0,\pi[$, where $\theta_0$ is the
unique solution of $\frac{\cos\theta}{\sin^2\theta}=2\pi$, we do have
$\phi\in]0,2\pi[$. Here, the field $X_\ell$ almost draws a complete turn
around itself when it goes around the parallels minus a point.

It is somehow surprising that there is just one volume functional concerning the $X_\ell$, unlike the case of $X_{\mathrm{m},k}$. They lack domain of definition and minimality equations, their volume is weak. Yet they are important.

A few important questions must be raised. Can one find a vector field with even less volume than $X_\ell$ in some equal volume subset of the 2-sphere? Is there a minimum of $\vol(X)$ per volume of its domain? What is the infimum and what does it depend on?

We end with a simple remark. It is possible to give a more concise definition involving the two types of vector fields $X_\mathrm{m},X_\ell$ briefly studied above.

Let the unit vector field $T=\frac{a}{r}\partial_\theta+\frac{b}{r\sin\theta}\partial_\phi$ with constant coefficients $a,b$ be defined on $M^\star$. This is of course the case $k=0$ meridian type vector field. The flow of $T=T_{a,b}$ integrates to a well-known family of curves, namely the loxodromes or rhumb lines. Indeed, these curves go across every meridian with a constant angle $\measuredangle(\partial_\theta,T)$. (Such is their original definition by Pedro Nunes in the XVIth century).

Now we define the vector fields of $T$-type as those $X_{T}$ which are parallel in the direction of $T_{a,b}$ for some $(a,b)\in\Sphere^1$ fixed:
\begin{equation}
 \na_TX_T=0.
\end{equation}
$T,X_{\mathrm{m},k},X_\ell$ are particular cases of vector fields of this type.  And from these we may define other just as well.

\hspace*{2mm}

We thank Olga Gil-Medrano for helpful conversations and Marta Barata for the drawings. Also, we thank the anonymous Referees' carefull readings and remarks which led to substantial improvements of this work.

\vspace*{13mm}

\ \\
\textsc{Rui Albuquerque}\ \ \ \textbar\ \ \
{\texttt{rpa@uevora.pt}}\\
Centro de Investiga\c c\~ao em Mate\-m\'a\-ti\-ca e Aplica\c c\~oes\\
Rua Rom\~ao Ramalho, 59, 671-7000 \'Evora, Portugal\\
The research leading to these results has received funding from Funda\c c\~ao para a Ci\^encia e a Tecnologia. Project Ref. UIDB/04674/2020.



\begin{thebibliography}{30}








\bibitem{Alb2021}
R. Albuquerque,
\emph{Calibrations for minimal volume vector fields in dimension 2},
\url{https://arxiv.org/abs/2109.01565v1} (2021), 7 pages.



\bibitem{BorGil2006}
V.~Borrelli, O.~Gil-Medrano,
\emph{A critical radius for unit Hopf vector fields on spheres},
Math. Ann. 334(4) (2006), 731--751.


\bibitem{BorGil2010}
V.~Borrelli, O.~Gil-Medrano,
\emph{Area-minimizing vector fields on round 2-spheres},
J. Reine Angew. Math. 640 (2010), 85--99.



\bibitem{BritoChaconJohnson}
F.~Brito, P.~Chac\'on and D.~Johnson,
\emph{Unit vector fields on antipodally punctured spheres: big index, big volume},
Bull. Soc. Math. France 136.1 (2008), 147--157.





\bibitem{BritoJackelineIcaroAdriana}
F. Brito, J. Conrado, I. Gonçalves and A.~Nicoli,
\emph{Area minimizing unit vector fields on antipodally punctured unit 2-sphere}, Comptes Rendus Math. 359, No. 10 (2021), 1225--1232.


\bibitem{BritoGomesGoncalves}
F.~Brito, A.~Gomes and I.~Gon\c calves,
\emph{Poincar\'e index and the volume functional of unit vector fields on punctured spheres},
manuscripta math. 161 (2020), 487--499.













\bibitem{GilMedranoLlinaresFusterSecond}
O.~Gil-Medrano and E.~Llinares-Fuster,
\emph{Second variation of volume and energy of vector fields. Stability of Hopf vector fields},
Math. Ann. 320, No. 3, 531-545 (2001).



\bibitem{GluckZiller}
H.~Gluck and W.~Ziller,
\emph{On the volume of a unit vector field on the three-sphere},
Comment. Math. Helv. 61 (1986), 177--192.





\bibitem{Hatcher}
A.~Hatcher, Algebraic topology. CUP, Cambridge, 2002.










\bibitem{Ped}
S.~Pedersen,
\emph{Volumes of vector fields on spheres}, Trans. Amer. Math. Soc. 336 (1993), 69--78.







\bibitem{Wieg}
G. Wiegmink,
\emph{Total bending of vector fields on Riemannian manifolds},
Math. Ann. 303, No. 2 (1995), 325--344.



\end{thebibliography}
\end{document}